\documentclass[smallextended,envcountsect]{svjour3}       
\usepackage{latexsym,amsmath,amsfonts,amscd,amssymb,verbatim,array,mathdots}

\smartqed  % flush right qed marks, e.g. at end of proof

\usepackage{chngcntr}
\counterwithin{equation}{section}

\DeclareMathOperator{\dom}{dom}
\DeclareMathOperator{\inte}{int}

\DeclareMathOperator{\Sol}{Sol}

\DeclareMathOperator{\TCP}{TCP}
\DeclareMathOperator{\CP}{CP}
\DeclareMathOperator{\Ba}{\mathbb{B}}
\DeclareMathOperator{\A}{\mathcal{A}}
\DeclareMathOperator{\B}{\mathcal{B}}
\DeclareMathOperator{\C}{\mathcal{C}}

\DeclareMathOperator{\Q}{\mathcal{Q}}

\DeclareMathOperator{\Ro}{\mathcal{R}_0}
\DeclareMathOperator{\Oo}{\mathcal{O}}

\DeclareMathOperator{\Sa}{\mathbb{S}}

\DeclareMathOperator{\rank}{rank}
\DeclareMathOperator{\R}{\mathbb{R}}

\begin{document}

%\begin{frontmatter}
\title{On the R$_0$--tensors and the solution map of tensor complementarity problems}
\author{Vu Trung Hieu}

\institute{Vu Trung Hieu \at Division of Mathematics, Phuong Dong University\\ 171 Trung Kinh Street, Cau Giay, Hanoi, Vietnam\\
E-mail: hieuvut@gmail.com}

\date{Received: date / Accepted: date}
\maketitle

\medskip
		
\begin{abstract} Our purpose is to investigate the local boundedness, the upper semicontinuity, and the stability of the solution map of tensor complementarity problems. To do this, we focus on the set of R$_0$--tensors and show that this set plays an important role in the investigation. Furthermore, by using a technique in semi-algebraic geometry, we obtain some results on the finite-valuedness and the lower semicontinuity of the solution map.
\end{abstract}
		

%\end{frontmatter}
	
\section{Introduction} 
The tensor complementarity problems were firstly introduced by Song and Qi in \cite{SongQi2014,SongQi2015}. These problems have attracted a lot of attention of researchers (see \cite{QCC2018} and the references given there).
 It is well-known that Huang and Qi have presented an explicit relationship between  $n$-person noncooperative games and tensor complementarity problems \cite{HuangQi17}. 
Structured tensors and different properties of their solution sets have been intensively investigated \cite{LQX2017,SongQi2016,SongYu2016,ZW2018}. The global uniqueness and solvability for tensor complementarity problems have been discussed in \cite{BHW2016} and \cite{LLW2017}. Methods and algorithms to solve a tensor complementarity problem have been interested by several authors \cite{DZ2018,LLW2017,XLX2017}. 

The involved function in a tensor complementarity problem is the sum of a homogeneous polynomial function and an arbitrary given vector. Thus, the tensor complementarity problem is a special case of the homogeneous complementarity problem which was mentioned in a work of Oettli and Yen \cite{OettliYen95}.  Besides,  the tensor complementarity problems also is a subclass of the polynomial  complementarity problems which have been recently introduced by Gowda \cite{Gowda16}.  
The tensor complementarity problem is a natural extension of the linear complementarity problem \cite{CPS1992}, hence, several properties of both problems are similar.  The boundedness, the continuity and the stability of the solution map of linear complementarity problems have been deeply investigated (see, e.g., \cite{CPS1992,Gowda1992,OettliYen95,Phung2002,Robinson79}).
We will investigate these properties of the solution map of tensor complementarity problems and simultaneously show that the set of R$_0$--tensors plays an important role in the investigation.

In the paper, firstly, we prove that the set of R$_0$--tensors is open. Accordingly, the local boundedness of the solution map is shown. Secondly, since the involved function in a tensor complementarity problem is polynomial, by using tools in semi-algebraic geometry,  we obtain the generic finite-valuedness of the solution map of tensor complementarity problems.  Consequently, a necessary condition for the lower semicontinuity of the solution map is obtained.  Furthermore, we show that the set of R$_0$--tensors is generic semi-algebraic in the set of real tensors. A lower bound for the dimension of the complement of R$_0$--tensors is obtained. Finally, the paper shows a closed relation between the upper semicontinuity of the solution map and the R$_0$ property of the involved tensors. A result on the stability of the solution map is introduced.

The organization of the paper is as follows.  Section 2 gives a brief introduction to tensor complementarity problems and semi-algebraic geometry. Main results are presented in the  next four sections. Section 3 investigates the openness of the set of R$_0$--tensors and the local boundedness of the solution map. Besides, this section proves the semi-algebraicity  and the genericity of the set R$_0$--tensors.
The finite-valuedness and the lower semicontinuity are discussed in Section 4. The last two sections give results on the upper semicontinuity and the stability of the solution map.

\section{Preliminaries}
In this section, we will recall some definitions, notations, and auxiliary results on tensor complementarity problems and semi-algebraic geometry.
\subsection{Tensor complementarity problems}
The scalar product of two vectors $x, y$ in the Euclidean space $\R^n$ is denoted by $\langle x,y\rangle$. Let $F: \R^n\to\R^n$ be a vector-valued function.
The \textit{nonlinear complementarity problems} defined by $F$ is the problem
$$\CP(F) \qquad {\text{Find}}\ \, x\in \R^n \ \; \text{such that}\ \, \ x\geq 0, \; F(x) \geq 0, \; \langle x,F(x)\rangle=0.$$
The solution set is denoted by $\Sol(F)$.

The following remark shows that a solution of a complementarity problem
can be characterized by using some Lagrange multipliers. 

\begin{remark}\label{KKT} A vector $x$ solves $\CP(F)$ if and only if there exists a vector $\lambda\in \R^n$ such that the following system is satisfied
	\begin{equation*}\label{KKT_formula}
	\left\lbrace \begin{array}{r}
	F(x)-\lambda=0,\\ 
	\langle \lambda,x \rangle=0,\\ \lambda\geq 0,\; x \geq 0.
	\end{array}\right.
	\end{equation*}
\end{remark}

To find the solution set of a complementarity problem, we will find the solutions on each pseudo-face of $\R^n_+$.  For every index set $\alpha\subset [n]=\{1,\dots,n\}$, we associate this set with the following \textit{pseudo-face}
$$K_{\alpha}=\left\lbrace x\in \R_+^n:x_i=0, \forall i\in\alpha; \; x_i >0, \forall i\in [n]\setminus\alpha\right\rbrace.$$
The subsets $K_{\alpha},\alpha\subset [n],$ is a finite disjoint decomposition of $\R^n_+$.
Therefore, we have 
\begin{equation}\label{quasiface}
\Sol(F)=\bigcup_{\alpha\subset [n]}\left[ \Sol(F)\cap K_{\alpha}\right].
\end{equation}

Throughout this paper, we assume that $m$ and $n$ are given integers, and  $m,n\geq2$. An $m$-th order $n$-dimensional \textit{tensor} $\A= (a_{i_1\cdots i_m})$ is a multi-array of real entries $a_{i_1\cdots i_m}\in\R$, where $i_j\in[n]$  and $j\in[m]$. The set of all real $m$-th order $n$-dimensional
tensors is denoted by $\R^{[m,n]}$. 
For any tensor $\A= (a_{i_1\cdots i_m})$,  the Frobenius norm of  $\A$ is defined and denoted as
$$\|\A\|:=\sqrt{\sum_{i_1,i_2,...,i_m=1}^{n}a^2_{i_1i_2\cdots i_m}}.$$
This norm can also be considered as a vector norm. So, the  norm of $(\A,a)$ in $\R^{[m,n]}\times\R^{n}$ can be defined as follows
$$\|(\A,a)\|:=\sqrt{\|\A\|^2+\sum_{i=1}^{n}a^2_i}.$$
Clearly, $\R^{[m,n]}$ is a real vector space of dimension $n^m$. Here, each tensor $\A\in\R^{[m,n]}$ is a real vector having ${n^m}$ components. 
If $m=2$ then $\R^{[2,n]}$ is the space of $n\times n$--matrices which is isomorphic to $\R^{n\times n}$. Note that if $\A= (a_{i_1\cdots i_m})$ and $\B= (b_{i_1\cdots i_m})$ are tensors in $\R^{[m,n]}$ then the sum of them is defined by
$$\A+\B=(a_{i_1\cdots i_m}+b_{i_1\cdots i_m}).$$

For any $x = (x_1,..., x_n)^T\in \R^n$, 
$\A x^{m-1}$ is a vector whose
$i$-th component defined by 
\begin{equation}\label{Axm_1}(\A x^{m-1})_i:=\sum_{i_2,...,i_m=1}^{n}a_{ii_2\cdots i_m}x_{i_2}\cdots x_{i_m}, \ \forall i\in[n],
\end{equation}
and $\A  x^m$ is a polynomial of degree $m$, defined by
$$\A x^{m}:=\langle x,Ax^{m-1} \rangle=\sum_{i_1,i_2,...,i_m=1}^{n}a_{i_1i_2\cdots i_m}x_{i_1}x_{i_2}\cdots x_{i_m}.$$
The polynomials $(Ax^{m-1})_i$ and $\A x^{m}$ are homogeneous of degree respectively $m-1$ and $m$, that is $\A(tx)^{m-1}=t^{m-1}( \A x^{m-1}) $ and $\A (tx)^{m}=t^m(\A x^m)$ for all $t\geq 0$ and $x\in\R^n$.

\begin{remark}\label{norm_beta}
By the continuity of the polynomial function $\A x^{m-1}$, if $U$ is an bounded set in $\R^n$ then there exists $\beta>0$ such that 
	$\|\A x^{m-1}\|\leq \beta\|\A\|$
	for all $x\in U$.
\end{remark}

Let $\A\in\R^{[m,n]}$ and $a \in\R^n$ be given. If $F(x)=\A x^{m-1}+ a$ then one says that $\CP(F)$ is a \textit{tensor complementarity problem} which defined by $\A$ and $a$. This problem and its solution set are denoted respectively by $\TCP(\A,a)$ and $\Sol(\A,a)$.
By definition, $x$ solves $\TCP(\A,a)$ if and only if
\begin{equation}\label{TCP}
x\geq 0, \ \A x^{m-1}+a\geq 0, \ \A x^{m}+\left\langle x,a \right\rangle =0.
\end{equation}
Clearly, the vector $0$ solves $\TCP(\A,a)$ for all $a\in\R^n_+$. The solution map of tensor complementarity problems is denoted and defined by
\begin{equation}\label{Sol}
\Sol:\R^{[m,n]}\times\R^{n}\rightrightarrows \R^n, \ (\A,a)\mapsto\Sol(\A,a).
\end{equation}

A subset $K\subset \R^n$ is called a cone \cite[p. 89]{Bernstein} if $\lambda >0$ and $x\in K$ then $\lambda x\in K$. The cone $K$ is bounded if and only if $K=\{0\}$.

\begin{remark} The solution set of $\TCP(\A,0)$ is a closed cone. Indeed, suppose that $x$ is a solution of $\TCP(\A,0)$. For each $t\geq0$, from \eqref{TCP}, one has
	$$tx\geq 0, \ \; \A(tx)^{m-1}=t^{m-1}(\A x^{m-1}) \geq 0, \ \; \A(tx)^{m}=t^{m}(\A x^m)=0.$$
	By definition, $tx$ solves $\TCP(\A,0)$.  This shows that $\Sol(\A,0)$ is a cone. The closedness of $\Sol(\A,0)$ is implied from the continuity of $F(x)=\A x^{m-1}$.
\end{remark}

Let us recall that $\A$ is an R$_0$--tensor (sometimes, we say that $\A$ is  R$_0$ or $\A$ has R$_0$ property) if $\Sol(\A,0)=\{0\}$. We denote $\Ro$ to be the set all $m$-th order $n$-dimensional R$_0$--tensors and $\mathcal{O}\in \R^{[m,n]}$ to be the zero tensor. The complement of $\Ro$ is denoted and defined by $$C(\Ro)=\R^{[m,n]}\setminus\Ro.$$ Clearly, $\mathcal{O}$ belongs to $C(\Ro)$ since $\Sol(\mathcal{O},0)=\R^n_+$.

\subsection{Semi-algebraic geometry}

Recall a subset in $\R^n$ is \textit{semi-algebraic} if it is the union of finitely
many subsets of the form
\begin{equation*}\label{basicsemi}
	\big\{x\in \R^n\,:\,f_1(x)=...=f_\ell(x)=0,\ g_{\ell+1}(x)<0,\dots,g_m(x)<0\big\},
\end{equation*}
where $\ell,m$ are natural numbers, and $f_1,\dots, f_\ell, g_{\ell+1},\dots,g_m$ are polynomials with real coefficients. The semi-algebraic property is preserved by taking finitely union, intersection, minus and taking closure of semi-algebraic sets. The well-known Tarski-Seidenberg Theorem states that the image of a semi-algebraic set under a linear projection is a semi-algebraic set.

There are some ways to define the dimension of a semi-algebraic set. Here, we choose the geometric approach which is represented in \cite[Corollary 2.8.9]{BCF98}. If $S\subset \R^n$ is a semi-algebraic set, then there exists a decomposition of $S$ into a disjoint union of semi-algebraic subsets \cite[Theorem 2.3.6]{BCF98} $$S=\bigcup_{i=1}^sS_i,$$ where each $S_i$ is semi-algebraically diffeomorphic to $(0,1)^{d_i}$.  
Here, let $(0,1)^{0}$ be a point,  $(0,1)^{d_i}\subset \R^{d_i}$ is the set of points $x=(x_1,\dots,x_{d_i})$ such that $x_j\in (0,1)$ for all $j=1,\dots,d_i$. The \textit{dimension} of $S$ is, by definition,
$$\dim(S):=\max\{d_1,...,d_s\}.$$		
The dimension is well-defined and not depends on the decomposition of $S$. We would like remind that if the dimension of a nonempty semi-algebraic set $S$ is zero  then  $S$ has finitely many points. 

\begin{remark}\label{dense}	For a subset $S$ of $\R^n$, the following notions are different: $S$ is full measure, meaning its complement $\R^n\setminus S$ has measure zero, and $S$ is topologically
	generic, meaning it contains a countable intersection of dense and
	open sets. However, for semi-algebraic sets, the
	following properties are equivalent (see, e.g., \cite[Lemma 2.3]{DHP16}): $S$ is dense in $\R^n$; $\R^n\setminus S$ has measure zero; and $\dim(\R^n\setminus S)<n$. If the semi-algebraic set $S$ satisfies one of three previous properties, then one says this set is generic.
\end{remark}

We will use the Tarski-Seidenberg Theorem in the third form in the next section. To present the theorem, we have to describe semi-algebraic sets via the language of first-order formulas. A \textit{first-order formula} (with parameters in $\R$) is
obtained by the induction rules \cite{Coste02}:
\begin{enumerate}
	\item[(i)] If $p \in\R[X_1, . . . , X_n]$, then $p > 0$ and $p= 0$ are first-order formulas;
	\item[(ii)] If $P,Q$ are first-order formulas, then ``$P$ \textit{and} $Q$'', ``$P$ \textit{or} $Q$'', ``\textit{not} $Q$'', which are denoted respectively by $P \wedge Q$, $P \vee Q$, and $\neg Q$, are first-order formulas;
	\item[(iii)] If $Q$ is a first-order formula, then $\exists X\, Q$ and $\forall X\, Q$, where $X$ is a variable ranging over $\R$, are first-order formulas.
\end{enumerate}
Formulas obtained by using only the rules (i) and (ii) are called \textit{quantifier-free	formulas}. A subset $S \subset \R^n$ is semi-algebraic if and only if there is a quantifier-free formula $Q_S(X_1,...,X_n)$ such that
$$(x_1,...,x_n) \in S\ \; \text{if and only if }\; Q_S(x_1,..., x_n).$$
In this case, $Q_S(X_1,...,X_n)$ is said to be a \textit{quantifier-free formula defining $S$.}

\begin{remark}%\label{Tar_Sei3} 
	The Tarski-Seidenberg Theorem in the third form \cite[Theorem~2.6]{Coste02}, says that if $Q(X_1,...,X_n)$ is a first-order formula, then the set $$S=\big\lbrace (x_1,...,x_n)\in\R^n: \ Q(x_1,...,x_n) \big\rbrace $$ is a semi-algebraic set. 
\end{remark}

\section{The set of R$_0$--tensors}
In this section, we prove that the set $\Ro$ of R$_0$--tensors is open in $\R^{[m,n]}$. Accordingly, the local boundedness of the solution map is shown. Moreover, we show that $\Ro$ is generic semi-algebraic. The dimension of the complement $C(\Ro)$ is discussed in the last subsection.

\subsection{Local boundedness of the solution map}

\begin{proposition}\label{open_cone} The set $\Ro$ of all R$_0$--tensors is open in $\R^{[m,n]}$.
\end{proposition}

\begin{proof} If the complement $C(\Ro)$ is closed, then the set $\Ro$ is open. So, we need only prove the closedness of $C(\Ro)$. Let $\{\A^k\}\subset C(\Ro)$ be a convergent sequence with $\A^k\to A$.
	On the contrary, we suppose that $\A\in\Ro$. For each $k$, $\Sol(\A^k,0)$ is unbounded. There exists an unbounded sequence  $\{x^k\}$ such that $x^k\in\Sol(\A^k,0)$ and $x^k\neq 0$ for each $k$. Without loss of generality we can assume that 
	$\|x^k\|^{-1}x^k\to\bar x$ and $\|\bar x\|=1.$
By definition, one has $$\A^k(x^k)^{m-1}\geq 0, \ \A^k(x^k)^m=0.$$ 
Dividing these ones by $\|x^k\|^{m-1}$ and $\|x^k\|^{m}$, respectively,  and taking $k\to+\infty$, we obtain 
	$\A(\bar x)^{m-1}\geq0$ and $\A(\bar x)^m=0.$
	It follows that $\bar x \in \Sol(\A,0)=\{0\}$. This contradicts to $\|\bar x\|=1$. Therefore, $\A$ must be in $\C(\Ro)$. Hence, $C(\Ro)$ is closed. The proof is complete. \qed
\end{proof}

\begin{remark}The set $\Ro$ is a cone in $\R^{[m,n]}$. Indeed, for any $t>0$, one has
	$$(t\A) x^{m-1}=t^{m-1} (\A x^{m-1}), \ (t\A)x^{m}=t^{m} (\A x^{m}).$$	
	This leads to $	\Sol(t\A,0)=\Sol(\A,0)$. Hence, 
	$\A\in \Ro$ if and only if  $t\A\in \Ro$. This implies that $\Ro$  is a cone. 
\end{remark}

The boundedness of solution sets of tensor complementarity problems and polynomial complementarity problems under the R$_0$ condition is mentioned in \cite{SongYu2016} and \cite{Gowda16}. 
Based on the openness of the set $\Ro$, we show that the the solution map is locally bounded.

Here, $\Ba(\Oo,\varepsilon)$  stands for the closed ball in $\R^{[m,n]}$ centered at $\Oo$ with radius $\varepsilon$. Similarly, $\Ba(0,\delta)$ is the closed ball in $\R^n$ centered at $0$ with radius $\delta$.
 
\begin{theorem}\label{bounded1} The following two statements are equivalent:
	\begin{description}
		\item[\rm (a)] The tensor $\A$ is R$_0$;
			\item[\rm (b)] There exists $\varepsilon>0$ such that the following set 
			$$S(\varepsilon,\delta):=\bigcup_{(\B,b)\in \Ba(\Oo,\varepsilon)\times \Ba(0,\delta)} \Sol(\A+\B, a+b),$$
			is bounded, for any $\delta>0$ and $a\in\R^n$.
	\end{description}
\end{theorem}
\begin{proof} $\rm(a) \Rightarrow (b)$ Let $\A$ be an R$_0$--tensor. Since the set $\Ro$ is open in $\R^{[m,n]}$, due to Proposition \ref{open_cone}, there exists  $\varepsilon>0$ such that $\A+\Ba(\Oo,\varepsilon)\subset\Ro$. 	
We suppose that there exists $a\in\R^n$ and $\delta>0$ such that the set $S(\varepsilon,\delta)$ is unbounded. Let $\{x^k\}$ be an unbounded sequence and $\{(\B^k,b^k)\}$ be a sequence in $\Ba(\Oo,\varepsilon)\times \Ba(0,\delta)$ satisfying
	\begin{equation}\label{xk_sol}
	x^k\neq 0, \ \|x^k\|^{-1}x^k\to\bar x,\ x^k\in\Sol( \A+\B^k,a+b^k).
	\end{equation}
By the compactness of the sets $\A+\Ba(\Oo,\varepsilon)$ and $a+\Ba(0,\delta)$,  we can assume that $$\A+\B^k\to\bar\A\in \A+\Ba(\Oo,\varepsilon), \ a+b^k\to \bar a\in a+\Ba(0,\delta).$$
From \eqref{xk_sol}, it is easy to check that $\bar x$ solves $\TCP(\bar\A,0)$. Since $\|\bar x\|=1$, $\bar\A$ is not R$_0$. This contradicts to the fact that $\bar\A$ belongs to $\A+\Ba(\Oo,\varepsilon)\subset\Ro$.  

$\rm(b) \Rightarrow (a)$ Suppose that there exists $\varepsilon>0$ such that $S(\varepsilon,\delta)$ is bounded for any $a\in\R^n$ and $\delta>0$. Clearly, one has $\Sol(\A,0)\subset S(\varepsilon,\delta)$. This implies that $\Sol(\A,0)$ is bounded, namely, $\A$ is an R$_0$--tensor. 
The proof is complete.\qed
\end{proof}

\begin{remark}%\label{SongYu_bounded} 
	The tensor $\A$ is R$_0$ if and only if $\Sol(\A,a)$ is bounded for every $a\in\R^n$ (see  \cite[Theorem 3.2]{SongYu2016}). Moreover, $\A$ is an R$_0$--tensor if and only if the set
$$\bigcup_{b\in \Ba(0,\delta)} \Sol(\A,a+b)$$ 
	is bounded, for every $a\in\R^n$ and $\delta>0$ \cite[Proposition 2.1]{Gowda16}. 
\end{remark}

Recall that the set-valued map $\Psi:\R^m\rightrightarrows\R^n$ is \textit{locally bounded} at $\bar x$ if there exists an open neighborhood $U$ of $\bar x$ such that the set $\cup_{x\in U} \Psi(x)$ is bounded.

\begin{corollary}%\label{bounded2} 
	The following two statements are equivalent:
	\begin{description}
		\item[\rm (a)] 	The tensor $\A$ is R$_0$;
		\item[\rm (b)] The solution map $\Sol$ is locally bounded at $(\A,a)$, for every $a\in\R^n$.
		\end{description}
\end{corollary}

\begin{proof}
$\rm(a) \Rightarrow (b)$ Suppose that $\A$ is an R$_0$--tensor. By Theorem \ref{bounded1}, there exists $\varepsilon>0$ such that, for every $a\in\R^n$, the following set is bounded
$$S(\varepsilon,\varepsilon)=\bigcup_{(\B,b)\in U}\Sol(\B,b),$$
where $U=(\A,a)+\Ba(\Oo,\varepsilon)\times \Ba(0,\varepsilon)$ is an open neighborhood of $(\A,a)$. This means that the map $\Sol$ is locally bounded at $(\A,a)$.

 $\rm(b) \Rightarrow (a)$ Suppose that the assertion in $\rm(b)$ is true. Taking $a=0$, there exists an open neighborhood $U$ of $(\A,0)$ such that 
$$\Sol(\A,0)\subset \bigcup_{(\B,b)\in U} \Sol(\B,b)$$
is bounded. It follows that $\Sol(\A,0)=\{0\}$. The proof is complete. \qed
\end{proof}

\subsection{Semi-algebraicity and genericity of $\Ro$}

\begin{proposition}%\label{sal}
	The set $\Ro$ is semi-algebraic in $\R^{[m,n]}$.
\end{proposition}
\begin{proof} Remind that $\A\in\Ro$ if $\Sol(\A,0)=\{0\}$. Since $0$ always belongs to $\Sol(\A,0)$, the set $\Ro$ can be described as follows:
	\begin{equation}\label{R0}
		\begin{array}{cl}
			\Ro&=\left\lbrace\A\in\R^{[m,n]}: \nexists x\in \R^n_+\setminus\{0\} \left( \left[\A x^{m-1}\geq 0\right]\wedge \left[ \A x^{m}=0  \right]\right)   \right\rbrace \medskip\\
			&=\left\lbrace\A\in\R^{[m,n]}: \forall x\in \R^n_+\setminus\{0\}\left(  \left[\A x^{m-1}\ngeq 0\right]\vee \left[ \A x^{m}\neq 0  \right] \right)  \right\rbrace.
		\end{array}
	\end{equation}
Because $\R^n_+$ and $\{0\}$ are semi-algebraic, the set $K=\R^n_+\setminus\{0\}$ is also a semi-algebraic set in $\R^n$. Let $Q_K(x)$ be the quantifier-free formula defining $K$. Since $\left( \A x^{m-1}\right)_i$, where $i=1,...,n,$ and $\A x^{m}$ are polynomials, the following  formulas 
	$$Q_1(\A,x)=\bigvee_{i=1}^{m}\left[ \left( \A x^{m-1}\right)_i<0\right]$$
	and 
	$$Q_2(\A,x)=\left[ \A x^{m}>0\right]  \vee\left[ \A x^{m}<0\right]$$
	are quantifier-free. From the last equation in \eqref{R0}, $\A\in \Ro$ if and only if $Q(\A)$, where $Q(\A)$ is the following first-order formula
	$$Q(\A):=\forall x \left( Q_K(x) \wedge\left[Q_1(\A,x) \vee Q_2(\A,x) \right] \right) .$$
	According to the Tarski-Seidenberg Theorem in the third form, $\Ro$ is a semi-algebraic set in $\R^{[m,n]}$. \qed
\end{proof}

Let $\Phi:X \to Y$ be a differentiable map between manifolds,  where $X\subset \R^m$ and $Y\subset \R^n$. A point $y \in Y$ is called a \textit{regular value} for $\Phi$ if either the level set $\Phi^{-1}(y)=\emptyset$ or the derivative
map 
$$D\Phi(x): T_xX\to T_yY$$
is surjective at every point $x \in \Phi^{-1}(y)$,
where $T_xX$ and $T_yY$ denote respectively the tangent spaces of $X$ at $x$ and of $Y$
at $y$. So $y$ is a regular value of $f$ if and only if $\rank D\Phi(x)=n$ for all $x\in \Phi^{-1}(y)$.

\begin{remark}\label{RegularLevel}
	Consider the differentiable semi-algebraic map $\Phi:X \to \R^n$ where $X\subset \R^n $. Assume that $y \in Y$ is a regular value of $\Phi$ and $\Phi^{-1}(y)$ is nonempty. According to the Regular Level Set Theorem \cite[Theorem 9.9]{Loring_2010}, one has  $\dim\Phi^{-1}(y)=0$. It follows that $\Phi^{-1}(y)$ has finite points.
\end{remark}

\begin{remark}\label{Sard_parameter}
	Let $\Phi : \R^p\times X\to \R^n$ be a differentiable semi-algebraic map, where $X\subset \R^n $. Assume that $y \in \R^n$ is a regular value of $\Phi$. According to the Sard Theorem with parameter \cite[Theorem 2.4]{DHP16}, there exists a generic
	semi-algebraic set $\Sa\subset\R^p$ such that, for every $p\in \Sa$, $y$ is a regular value of the map
	$\Phi_p:X\to Y$ with $\Phi_p(x)=\Phi(p,x)$.
\end{remark}

\begin{theorem}\label{generic_1}  The set $\Ro$ of all R$_0$--tensors  is generic in $\R^{[m,n]}$.
\end{theorem}
\begin{proof} We will show that there exists a generic semi-algebraic set $\Sa\subset\R^{[m,n]}$ such that $\Sol(\A,0)=\{0\}$ for all $\A\in\Sa$. 	Indeed, let $K_{\alpha}\neq\{0\}$ be a given pseudo-face of $\R^n_+$. 
To avoid confusion, we only consider the case $\alpha=\{1,...,\ell\}$, where $\ell<n$, because other cases can be treated similarly.	
Then, if $x\in K_{\alpha}$ then $x_{\ell+1}\neq 0$.
	We consider the function
	$$\Phi_\alpha:\R^{[m,n]}\times K_{\alpha}\times \R^{\ell} \ \to \ \R^{n+\ell}$$
	which is defined by 
\begin{equation}\label{Phi_alpha}
	\Phi_\alpha(\A,x,\lambda_\alpha)=\left( \A x^{m-1}-\lambda, x_{\alpha}\right)^T,
	\end{equation}
where $x_{\alpha}=(x_{1},...,x_{\ell})$,  $\lambda_{\alpha}=(\lambda_{1},...,\lambda_{\ell})$ and $\lambda=(\lambda_{1},...,\lambda_{\ell},0,...,0)\in\R^n$. The Jacobian matrix 
	of $\Phi_\alpha$ is determined as follows
	$$D\Phi_\alpha=\left[ \begin{array}{c|c|c}
	\ D_{\A}(\A x^{m-1}-\lambda) \ & \ D_{x}(\A x^{m-1}-\lambda) \ & \ D_{\lambda_{\alpha}}(\A x^{m-1}-\lambda) \\  
	\hline
	D_{\A}(x_{\alpha}) & D_{x}(x_{\alpha})  & D_{\lambda_{\alpha}}(x_{\alpha})  \\ 
	\end{array}\right].$$
	
	We claim that the rank of $D\Phi_\alpha$ is $n+\ell$ for all $x\in K_\alpha$. Indeed,  it is easy to check that the rank of $D_{x}(x_{\alpha})$ is $\ell$. Therefore, if we prove that the rank of $D_{\A}(\A x^{m-1}-\lambda)$ is $n$ then the claim follows. Clearly, one has
	$$D_{\A}(\A x^{m-1}-\lambda)=\begin{bmatrix}
	Q_{1}	&0_{1\times n}  & \cdots &0_{1\times n}   \\ 
	0_{1\times n} 	&Q_{2}  & \cdots &0_{1\times n}   \\ 
	&  & \ddots &  \\ 
	0_{1\times n} 	&0_{1\times n}   & \cdots & Q_{n}
	\end{bmatrix},$$
	$0_{1\times n}$ is the zero $1\times n$--matrix, $Q_{i}$ is an $1\times n$--matrix. From \eqref{Axm_1} and \eqref{Phi_alpha}, for each $i\in[n]$, we conclude that $Q_{i}$ is a nonzero matrix because the $(\ell+1)$--th entry of $Q_{i}$ is given by
	$$\dfrac{\partial(\A x^{m-1}-\lambda)_i}{\partial a_{i\ell\cdots \ell}} =x_{\ell}^{m-1}\neq 0.$$
	This shows that $\rank D_{\A}(\A x^{m-1}-\lambda)=n$.
	
	Therefore,
	$0\in \R^{n+\ell}$ is a regular value of $\Phi_\alpha$. According to Remark \ref{Sard_parameter}, there exists a generic semi-algebraic set  $\Sa_{\alpha}\subset \R^{[m,n]}$,  such that if $\A\in \Sa_{\alpha}$ then $0$ is a regular value of the map
	$$\Phi_{\alpha,\A}:K_{\alpha}\times \R^{\ell} \to \R^{n+\ell}, \ \Phi_{\alpha,\A}(x,\lambda_\alpha) =\Phi_\alpha(\A,x,\lambda_\alpha).$$
	Remark \ref{RegularLevel} claims that if the set $\Omega(\alpha,\A):=\Phi^{-1}_{\alpha,\A}(0)$ 
	is nonempty then it is a zero dimensional semi-algebraic set. Hence, $\Omega(\alpha,\A)$ is a finite set. Moreover, from \eqref{Phi_alpha} and Remark \ref{KKT}, one has
	$$\Sol(\A,0)\cap K_{\alpha}=\pi(\Omega(\alpha,\A)),$$
	where $\pi$ is the projection $\R^{n+\ell} \to \R^n$ which is defined by $\pi(x,\lambda_{\alpha}) = x$. Thus, the cardinality of $\Sol(\A,0)\cap K_{\alpha}$ is finite.
	
	If $K_{\alpha}=\{0\}$, i.e. $\alpha=[n]$, then $\Sol(\A,0)\cap K_{\alpha}=\{0\}$. By the finite decomposition in \eqref{quasiface}, $\Sol(\A,0)$  is a finite set. 

By setting $\Sa:=\cap_{\alpha\subset [n]}\Sa_{\alpha}$, for any $\A$ in $\Sa$, the cardinality of $\Sol(\A,0)$ is finite. Since $\Sol(\A,0)$ is a cone, one has $\Sol(\A,0)=\{0\}$. This follows that $\Sa\subset \Ro$, consequently, $\Ro$ is generic in $\R^{[m,n]}$. The proof is completed. \qed
\end{proof}

\begin{remark} Theorem 6 in \cite{OettliYen95} asserts that the set of all R$_0$--matrices is dense in $\R^{n\times n}$. This is a special case of Theorem \ref{generic_1} when $m=2$. 
\end{remark}

\subsection{The dimension of $C(\Ro)$}
From Remark \ref{dense} and Theorem \ref{generic_1}, the complement $C(\Ro)$ is thin in the set of real $m$-th order $n$-dimensional tensors. A natural question is: \textit{How  $C(\Ro)$ is thin in $\R^{[m,n]}$}? The dimension of $C(\Ro)$ tell us about the thinness of this set. The following theorem gives a rough lower estimate for $\dim C(\Ro)$.

\begin{theorem}%\label{dim_alpha} 
	The dimension of the semi-algebraic set $C(\Ro)$ satisfies the following inequalities
	$$(n-1)^{m}\leq\dim C(\Ro)\leq n^m-1.$$
\end{theorem}
\begin{proof} The second inequality immediately follows from Theorem \ref{generic_1} and Remark \ref{dense}. To prove the first inequality, let $\alpha\subset [n]$ be given with $\alpha\neq[n]$, 	
	we consider the set
	$$\Sa_{\alpha}=\left\lbrace \A=(a_{i_1i_2\cdots i_m})\in\R^{[m,n]}:a_{i_1i_2\cdots i_m}=0, \  \forall i_j\in[n]\setminus\alpha \right\rbrace.$$	
	It follows that $\Sa_{\alpha}$ is a subspace of $\R^{[m,n]}$ with the dimension 
	$|\alpha|^{m}$. Hence, $\Sa_{\alpha}$ is semi-algebraic.
	Let us denote by 
	$\bar K_{\alpha}$ the face
	$$\bar K_{\alpha}=\left\lbrace x\in \R_+^n:x_i=0, \forall i\in\alpha; \; x_i \geq 0, \forall i\in [n]\setminus\alpha\right\rbrace.$$
	A trivial verification shows that $\bar K_{\alpha}\subset\Sol(\A,0)$ for all $\A\in \Sa_{\alpha}$. We conclude that the subspace $\Sa_{\alpha}$ is a subset of $C(\Ro)$. Thus, one has
	$$|\alpha|^{m}=\dim\Sa_{\alpha}\leq\dim C(\Ro).$$
	Take $\alpha=\{2,...,n\}$, one has $|\alpha|=n-1$. The first inequality is obtained. \qed
\end{proof}

\section{Lower semicontinuity of solution maps}
We will prove that the solution map of tensor complementarity problems is finite-valued on a generic semi-algebraic set in the parametric space. Consequently, a necessary condition for the lower semicontinuity of the solution map is given.

Now we recall some notions in set-valued analysis. The set-valued map $\Psi:\R^m\rightrightarrows \R^n$ is \textit{finite-valued} on $S\subset \R^m$ if the cardinality of the image $\Psi(x)$ is finite, namely $|\Psi(x)|<+\infty$, for all $x\in S$. The map $\Psi$ is  \textit{upper semicontinuous} at $x\in \R^m$ iff for any open set $V\subset Y$ such that $\Psi(x)\subset V$ there exists a neighborhood $U$ of $x$ such that $\Psi(x')\subset V$ for all $x'\in U$. If $\Psi$ is upper semicontinuous at every $x\in \R^m$ then $\Psi$ is said that to be upper semicontinuous on $\R^m$. The map $\Psi$ is \textit{lower semicontinuous} at $\bar x$ (see \cite[Proposition 2.1.17]{FaPa03}) if $\Psi(\bar x)=\liminf_{x\to \bar x}\Psi(x)$, where
$$\liminf_{x\to \bar x}\Psi(x)=\left\lbrace u\in\R^n\;: \;\forall x^k\to \bar x, \; \exists u^k\to u \; \text{ with } \; u^k\in \Psi(x^k)\right\rbrace.$$
If $\Psi$ is lower semicontinuous at every $x\in X$ then $\Psi$ is said that to be lower semicontinuous on $X$.

\begin{remark}\label{Coste413} The number of connected components of $\Sol(\A,a)$ does not excess 
	$	\chi=d(2d-1)^{5n}$,	where $d=\max\left\{2, m-1\right\}.$
	Indeed, let $\Omega$ be the set of all $(x,\lambda)\in \R^n\times\R^n$ such that the following conditions are satisfied
	$$
	\left\lbrace \begin{array}{r}
	\A x^{m-1}+q-\lambda=0,\\ 
	\langle \lambda,x \rangle=0,\\ \lambda\geq 0,x \geq 0.
	\end{array}\right.
	$$Clearly, $\Omega$ is a semi-algebraic set determined by $3n+1$ polynomial equations and inequalities in $2n$ variables, whose degrees do not exceed the number $d$. According to \cite[Proposition 4.13]{Coste02}, the number of connected components of $\Omega$ does not excess the  number $\chi$. By the definition of $\Omega$, one has $\Sol(\A,a)=\pi(\Omega),$
	where $\pi$ is the projection 
	$$\R^{n+n} \to \R^n, \ \pi(x,\lambda) = x.$$ 
By the continuity of $\pi$, the number of connected components of $\Sol(\A,a)$ also does not excess $\chi$. \end{remark}

In the following proposition, we consider the finite-valuedness of two solution maps of the tensor complementarity problems $\Sol$ given by  \eqref{Sol} and $\Sol_{\A}$ defined by
\begin{equation}\label{Sol_A}
\Sol_{\A}:\R^{n}\rightrightarrows \R^n, \ a\mapsto\Sol_{\A}(a)=\Sol(\A,a),
\end{equation}
where $\A$ is given.

\begin{proposition}\label{generic_3}
 There exists a generic semi-algebraic set $\Sa\subset\R^{[m,n]}\times\R^{n}$  such that the map $\Sol$ is finite-valued on $\Sa$.
\end{proposition}
\begin{proof} To prove the assertion, we apply the argument in the proof of Theorem~\ref{generic_1} again, the only difference
	being in the analysis of the function
	$$\Phi_\alpha:\R^{[m,n]}\times\R^n\times K_{\alpha}\times \R^{\ell} \ \to \ \R^{n+\ell}$$
	which defined by
	$$\Phi_\alpha(\A,a,x,\lambda_\alpha)=\left( \A x^{m-1}+a-\lambda, x_{\alpha}\right) ^T.$$
	Note that, since $D_{\A} \Phi_\alpha$ has rank $n$, the rank of $D\Phi_\alpha$ is $n+\ell$ for $x\in K_\alpha\neq\{0\}$, and the proof is complete. \qed
\end{proof}

\begin{remark}\label{generic_5} Let $\A$ be given. There exists a generic semi-algebraic set $\Sa_{\A}\subset\R^{n}$ such that the map $\Sol_{\A}$  is finite-valued on $\Sa_{\A}$.
This property is implied from \cite[Theorem 3.2]{LLP2018} with the note that $\R^n_+$ is a semi-algebraic set satisfying the linearly independent constraint qualification.
\end{remark}

\begin{theorem}\label{generic_4} If the solution map $\Sol$ is lower semicontinuous at $(\A,a)$, then $\Sol(\A,a)$ has finite elements. Hence, if $\dim\Sol(\A,a)\geq 1$, then $\Sol$ is not lower semicontinuous at $(\A,a)$.
\end{theorem}
\begin{proof} According to Proposition \ref{generic_3}, there exists a  generic set $\Sa$ in $\R^{[m,n]}\times\R^n$ such that $\Sol$ is finite-valued on $\Sa$. By the density of $\Sa$, there exists a sequence
	$\{(\A^k,a^k)\}\subset \Sa$ such that $(\A^k,a^k)\to (\A,a)$. From Remark \ref{Coste413}, $\Sol(\A^k,a^k)$ has finitely points and $|\Sol(\A^k,a^k)|\leq\chi$. Since $\Sol$  is lower semicontinuous, one has
	$$\Sol(\A,a)=\liminf_{k\to+\infty}\Sol(\A^k,a^k).$$
	It follows that	$|\Sol(\A,a)|\leq\chi$. The first assertion is proved. The second assertion follows the first one.
\qed
\end{proof}

\begin{example}\label{example_1} 
	Consider the tensor complementarity  problem $\TCP(\A,a)$ where $\A\in\R^{[3,2]}$ given by  $a_{111}=a_{122}=-1$, $a_{211}=a_{222}=-1$ and all other $a_{i_1i_2i_3}=0$. Obviously, one has
	$$\A x^{m-1}+q=\begin{bmatrix}
	-x_1^2-x_2^2\\
	-x_1^2-x_2^2
	\end{bmatrix}+\begin{bmatrix}
	a_1\\
	a_2
	\end{bmatrix},$$
	where the parameters $a_1,a_2\in\R$. 	From Remark \ref{KKT}, $x\in\Sol(\A,a)$ if and only if there exists $\lambda\in\R^2$ such that
$$\begin{bmatrix}
-x_1^2-x_2^2\\
-x_1^2-x_2^2
\end{bmatrix}+\begin{bmatrix}
a_1\\
a_2
\end{bmatrix}-\begin{bmatrix}
\lambda_1\\
\lambda_2
\end{bmatrix}=
\begin{bmatrix}
0\\
0
\end{bmatrix}.$$	
An easy computation shows that
	$$\Sol(\A,a)=\left\{\begin{array}{cl}
	\left\lbrace(0,0),(0,\sqrt{a_2})\right\rbrace  & \text{ if } 0\leq a_2<a_1, \\
	\left\lbrace(0,0),(\sqrt{a_1},0)\right\rbrace & \text{ if } 0\leq a_1< a_2, \\
	\left\lbrace(0,0)\right\rbrace \cup S_{a_1}& \text{ if } 0\leq a_1=a_2, \\
	\emptyset & \text{ if otherwise},\\
	\end{array}\right.$$
	where $$S_{a_1}=\{(x_1,x_2):x_1^2+x_2^2=a_1,x_1\geq 0, x_2\geq 0\}.$$ 
	Clearly, $\Sol_{\A}(a)$ is finite-valued for all $a\in\Sa$, where
	$$\Sa=\R^2\setminus \{a\in\R^2: 0 < a_1=a_2\}.$$
The set $\Sa$ is generic semi-algebraic set in $\R^2$. Moreover, since $\dim S_{a_1}=1$ with $a_1>0$, according to Theorem \ref{generic_4}, the map $\Sol$ is not lower semicontinuous at $(A,a)$ where $a\in\R^2$ with $0 < a_1=a_2$.
\end{example}

\section{Upper semicontinuity of the solution map}
This section establishes a closed relationship between the R$_0$ property and the upper semicontinuity of the solution map of tensor complementarity problems. Furthermore, two results on the single-valued continuity of the solution map $\Sol_{\A}$ are obtained.

\subsection{Necessary and sufficient conditions}
\begin{proposition}\label{usc_1}  If $\A$ is an R$_0$--tensor, then the map $\Sol$	is upper semicontinuous at $(A,a)$, for every $a\in\R^n$ satisfying $\Sol(\A,a)\neq\emptyset$.
\end{proposition}
\begin{proof} Suppose that $\A$ is an R$_0$--tensor but there is $a\in\R^n$ such that $\Sol(\A,a)\neq \emptyset$ and $\Sol$ is not upper semicontinuous at $(\A,a)$. There exists a nonempty open set $V$ containing $\Sol(\A,a)$, a sequence $\{(\A^k,a^k)\}$, and a sequence $\{x^k\}$ satisfying $(\A^k,a^k)\to (\A,a)$ and
	\begin{equation}\label{V_open0}
	x^k\in\Sol(\A^k,a^k)\setminus V.
	\end{equation}
According to Theorem \ref{bounded1}, there exists $k_0$ such that $\{x^k, k\geq k_0\}$ is a bounded subsequence. So, the sequence $\{x^k\}$ is bounded. Without loss of generality we can assume that  $x^k\to \bar x$. It is easy to check that  $\bar x$ solves $\TCP(\A,a)$. It follows that $\bar x\in V$. By the openness of $V$ and \eqref{V_open0}, one has $\bar x\notin V$. We obtain a contradiction. Therefore, the map $\Sol$ is upper semicontinuous at $(\A,a)$. \qed \end{proof}

\begin{corollary}\label{SolA_usc} If $\A$ is an R$_0$--tensor, then the map $\Sol_{\A}$	 is upper semicontinuous at $a$, for every $a\in\R^n$ satisfying $\Sol(\A,a)\neq\emptyset$.
\end{corollary}
\begin{proof} Suppose that $\A$ is an R$_0$--tensor and  $\Sol(\A,a)$ is nonempty. 	
	From Proposition \ref{usc_1}, the map $\Sol$ is upper semicontinuous at $(A,a)$. For any open neighborhood $V$ of $\Sol(A,a)$, there exists an open neighborhood $U$ of $(A,a)$ such that $(\B,b)\in U$ then $\Sol(\B,b)\subset V$. Consider the map $$\varphi:\R^{[m,n]}\times\R^n\rightrightarrows \R^n, \ (\B,b)\mapsto b.$$
Clearly, $\varphi$ is surjective, continuous and linear. According to the Theorem Open Mapping \cite[Theorem 2.11]{Rudin91}, $\varphi(U)$ is an open neighborhood of $a$. By definition, $\Sol_{\A}$	 is upper semicontinuous at $a$. \qed	
\end{proof}

\begin{example}
	Consider the tensor complementarity  problem $\TCP(\A,a)$ given in Example \ref{example_1}. 
	One has $\Sol_{\A}(0)=\{0\}$, so $\A$ is an R$_0$--tensor. From Corollary \ref{SolA_usc}, the solution map $\Sol_{\A}$ is upper semicontinuous on $\R^2_+$.
\end{example}

\begin{remark} The inverse assertion in Proposition \ref{SolA_usc} is not true.  Indeed, choose $\A=\Oo\in\R^{[3,2]}$, one has
	\begin{equation*}\label{SolO}
	\Sol_{\Oo}(a_1,a_2)=\left\{\begin{array}{ccc}
	\R^2_+ & \text{ if } &a_1=0,a_2=0, \\
	\R_+\times\{0\}&  \text{ if } & a_1=0,a_2>0,  \\
		\{0\}\times\R_+ &  \text{ if } & a_1>0,a_2=0, \\
	\{(0,0)\} &  \text{ if } & a_1>0,a_2>0,\\	
	\emptyset & \text{ if } & \text{ otherwise.}
	\end{array}\right.
	\end{equation*}
It is easy to check that $\Sol_{\Oo}$ is upper semicontinuous on $\dom\Sol_{\Oo}=\R^2_+$, but $\Oo$ has not R$_0$ property.
\end{remark}

\begin{proposition}\label{usc_2} Assume that $\Sol(\A,a)$ is nonempty and bounded. If  the map $\Sol$ is upper semicontinuous at $(\A,a)$, then  $\A$ is a R$_0$--tensor.
\end{proposition}
\begin{proof}
	Suppose that $\Sol(\A,0)\neq\{0\}$ and $y\in\Sol(\A,0)$ with $y\neq 0$. According to Remark \ref{KKT}, there exists $\lambda\in\R^n$ such that 
	\begin{equation}\label{KKT_H0}
	\left\lbrace \begin{array}{r}
\A y^{m-1}-\lambda=0,\\ 
	\langle \lambda, y \rangle=0,\\ \lambda\geq 0, \; y \geq 0.
	\end{array}\right. 
	\end{equation}
	For each $t\in(0,1)$, we take $y_t=t^{-1}y$ and $ \lambda_t=t^{-(m-1)}\lambda$. We will show that for every $t$ there exists $\A_t\in\R^{[m,n]}$ with $\A_t\to \A$ when $t\to 0$ and the following system is satisfied
	\begin{equation}\label{KKT_Ht}
	\left\lbrace \begin{array}{r}
	\A_t (y_t)^{m-1}+{q}-\lambda_t=0,\\ 
	\langle \lambda_t,y_t \rangle=0,\\ \lambda_t\geq 0, \; y_t \geq 0.
	\end{array}\right. 
	\end{equation}
	
	Since $y=(y_1,...,y_n)\neq 0$, there exists $y_{\ell}\neq 0$, so one has $y^{m-1}_{\ell}\neq 0$. Taking $\Q\in\R^{[m,n]}$ such that $$\Q x^{m-1}=\left( q_1x_i^{m-1},...,q_nx_i^{m-1}\right),$$ 
	where $q_j=-a_j/y_{\ell}^{m-1}$ for $j=1,...,n$. It is clear that $\Q y^{m-1}+ a=0$. We take 
	$\A_t= \A+t\Q$ and prove that the system \eqref{KKT_Ht} is true. Indeed, the last two  inequalities in \eqref{KKT_Ht} are obvious. Consider the left-hand side of the first equation in~\eqref{KKT_Ht}, from \eqref{KKT_H0} we have
	$$ \begin{array}{cl}
	\A_t(y_t)^{m-1}+ a-\lambda_t&=(\A_t=\A+t\Q)(t^{-1}y)^{m-1} +a- t^{-(m-1)}\lambda\\ 
	&=t^{-(m-1)}(\A y^{m-1}-\lambda)+(\Q y^{m-1}+a) \\ 
	&= 0.
	\end{array} $$
	The second equation  in~\eqref{KKT_Ht} is obtained by
	$$\langle \lambda_t,y_t \rangle=\langle t^{-(m-1)}\lambda,t^{-1}y \rangle=t^{-m}\langle \lambda,y \rangle=0.$$
	According to Remark \ref{KKT}, the system \eqref{KKT_Ht} leads to $y_t \in \Sol(\A_t,a)$. Remind that this assertion is true for all $t\in (0,1)$.
	
	Since $\Sol(\A,a)$ is nonempty bounded, let $V$ be a nonempty bounded open set containing $\Sol(\A,a)$. By the upper semicontinuity of $\Sol$ at $(\A,a)$, there exists $\delta>0$ such that $\Sol(\B,b)\subset V$ for all $(\B,b)\in\R^{[m,n]}\times\R^n$ satisfying $\|(\B,b)-(\A,a)\|<\delta$. Taking $t$ small enough such that $\|(\A_t,a)-(\A,a)\|<\delta$, we have $\Sol(\A_t,a)\subset V$. So, $y_t\in V$ for every $t>0$ sufficiently small. This is impossible, because $V$ is bounded and $y_t =t^{-1}y \to \infty$ as $t\to 0$. The proof is complete. \qed
\end{proof}

The main result of this section is shown in the following theorem.

\begin{theorem}%\label{ucs_3}  
	The following two statements are equivalent:
	\begin{description}
		\item[\rm (a)] The tensor $\A$ is R$_0$;
		\item[\rm (b)] The map $\Sol$ is upper semicontinuous at $(\A,a)$, for every $a\in\R^n$ satisfying $\Sol(\A,a)\neq\emptyset$.
	\end{description}
\end{theorem}
\begin{proof} Proposition \ref{usc_1} shows that $\rm(b)$ follows $\rm(a)$. Hence, we need only prove the direction $\rm(b) \Rightarrow (a)$.  Note that $0\in\Sol(\A,a)\neq\emptyset$ for every $a\in\R^n_+$. According to Remark \ref{generic_5}, then there exists $q\in\R^n_+$ such that $\Sol(\A,a)$ is bounded. By assumptions,  the map $\Sol$ is upper semicontinuous at $(\A,a)$. Proposition~\ref{usc_2} says that $\A$ is  R$_0$. The proof is complete. \qed
\end{proof}

\subsection{Single-valued continuity}
Recall that $\TCP(\A,a)$ is said to have the \textit{GUS-property} if $\TCP(\A,a)$ has a unique solution for every $a\in\R^n$. Some special structured tensors which have GUS-property are shown in \cite{BHW2016,LLW2017}. A new property of the GUS-property of tensor complementarity problems is given in the following theorem. 

\begin{theorem}%\label{GUS}
	 If $\TCP(\A,a)$ has the GUS-property, 
	then the map $\Sol_{\A}$ is single-valued continuous on $\R^n$.
\end{theorem}
\begin{proof} By assumptions, $\TCP(\A,0)$ has a unique solution. This implies that $\A$ is an R$_0$--tensor. Corollary \ref{SolA_usc} shows that $\Sol_{\A}$ is upper semicontinuous on $\R^n$. Therefore, $\Sol_{\A}$ is single-valued continuous on $\R^n$. \qed
\end{proof}

\begin{example}%\label{example_2}
	Consider the tensor complementarity  problem $\TCP(\A,a)$ where $\A\in\R^{[3,2]}$ given by  $a_{111}=a_{222}=1$ and all other $a_{i_1i_2i_3}=0$. Obviously, one has
$$\A x^{m-1}+q=\begin{bmatrix}
	x_1^2\\
	x_2^2
	\end{bmatrix}+\begin{bmatrix}
	a_1\\
	a_2
	\end{bmatrix},$$
	where the parameters $a_1,a_2\in\R$. 	An easy computation shows that
	$$\Sol_{\A}(a_1,a_2)=\left\{\begin{array}{cl}
	\left\lbrace(\sqrt{-a_1},\sqrt{-a_2})\right\rbrace & \text{ if } a_1<0, a_2<0 \\
	\left\lbrace (0,\sqrt{-a_2} )\right\rbrace & \text{ if } a_1\geq 0, a_2<0 \\
	\left\lbrace (\sqrt{-a_1},0)\right\rbrace & \text{ if } a_1< 0, a_2\geq0 \\
	\left\lbrace (0,0)\right\rbrace & \text{ if } a_1\geq 0, a_2\geq0 \\
	\end{array}\right.$$
The problem $\TCP(\A,a)$ has the GUS-property, the domain of $\Sol_{\A}$ is $\R^2$ and  $\Sol_{\A}$ is single-valued and continuous on $\R^2$. 
\end{example}

Recall that a tensor $\A$ is  \textit{copositive}
if  $\A x^m\geq 0$ for all $x \geq 0$.  A function $F:\R^n\to \R^n$ is 
\textit{monotone} on $X\subset\R^n$ if for all $x,y\in X$ the following inequality is satisfied
\begin{equation}\label{condition_monotone}
\left\langle F(y)-F(x),y-x\right\rangle \geq0.\end{equation}
If $F(x)=\A x^{m-1}$ is monotone on $\R^n_+$ then $\A$  is copositive. Indeed, by taking  $y=0$ in \eqref{condition_monotone}, $\A$ satisfies the copositive condition. %Furthermore, if $F(x)=\A x^{m-1}$ is monotone on $\R^n_+$ then $\Sol(\A,a)$ is convex for all $a\in\R^n$.

\begin{remark}%\label{Gowda_cop1} 
	If the R$_0$--tensor $\A$ is copositive, then $\Sol(\A,a)$ is nonempty for every $q\in\R^n$ \cite[Corollary 7.2]{Gowda16}.
\end{remark}

\begin{theorem}%\label{usc_4} 
	Assume that $\A$  is an R$_0$--tensor. If $F(x)=\A x^{m-1}$ is monotone on $\R^n_+$, then the map $\Sol_{\A}$ is single-valued continuous on a generic semi-algebraic set in $ \R^{n}$. 
\end{theorem}
\begin{proof}  By the copositity and the R$_0$ property of $\A$, according to Corollary~7.2 in \cite{Gowda16}, one has $\Sol_{\A}(a)\neq\emptyset$ for all $a\in\R^n$.
From Theorem \ref{generic_4}, there exists a generic semi-algebraic set $\Sa\subset \R^{n}$ such that $\Sol_{\A}$ is finite-valued on $\Sa$.  
	
For every $a\in\R^n$, by the monotonicity of $F$, $F+a$ also is monotone. It follows that $\Sol_{\A}(a)$ is convex \cite[Theorem 2.3.5]{FaPa03}. Since $\Sol_{\A}(a)$ is nonempty and has finite points, it has an unique point. This implies that $\Sol_{\A}$ is single-valued on $\Sa$. Moreover, $\A$  is an R$_0$--tensor, Remark \ref{SolA_usc} implies that $\Sol_{\A}$ is upper semicontinuous on $\Sa$. From what has already been shown,  $\Sol_{\A}$ is single-valued continuous on $\Sa$. \qed
\end{proof}

\begin{example}%\label{example_3} 
	Consider the tensor complementarity  problem $\TCP(\A,a)$ where $\A\in\R^{[3,2]}$ given by  $a_{111}=a_{122}=1$, $a_{211}=a_{222}=1$ and all other $a_{i_1i_2i_3}=0$. Obviously, one has
	$$F(x)=\A x^{m-1}+a=\begin{bmatrix}
	x_1^2+x_2^2\\
	x_1^2+x_2^2
	\end{bmatrix}+\begin{bmatrix}
	a_1\\
	a_2
	\end{bmatrix},$$
	where the parameters $a_1,a_2\in\R$. The Jacobian matrix of $F$ is positive semidefinite on $\R^2_+$. Hence, $F$ is monotone on $\R^2_+$. From Remark \ref{KKT}, an easy computation shows that
	$$\Sol_{\A}(a_1,a_2)=\left\{\begin{array}{cl}
	\left\lbrace(0,\sqrt{-a_2})\right\rbrace  & \text{ if } a_2<0, a_2\leq a_1, \\
	\left\lbrace(\sqrt{-a_1},0)\right\rbrace & \text{ if } a_1<0, a_1\leq a_2, \\
	\left\lbrace(0,0)\right\rbrace& \text{ if } 0\leq a_1,0\leq a_2,\\
	S_{-a_1}& \text{ if } a_1=a_2<0, \\
	\end{array}\right.$$
	where $$S_{-a_1}=\{(x_1,x_2):x_1^2+x_2^2=a_1,x_1\geq 0, x_2\geq 0\}, \ a_1<0.$$ 
The tensor $\A$ is R$_0$ since $\Sol_{\A}(0,0)=\{(0,0)\}$. The map $\Sol_{\A}$ is single-valued continuous on the generic semi-algebraic set $\Sa$, where
$$\Sa=\R^2\setminus \{q\in\R^2: a_1=a_2<0\}.$$
\end{example}

\section{Stability of the solution map}
This section discusses on the stability of the solution map of tensor complementarity problems. We will show that the map $\Sol$ is locally upper-H\"{o}lder when the involved tensor is R$_0$. In addition, if the tensor is copositive then one obtains a result on the stability of the solution map.

Recall that the map $\Sol_{\A}$ defined in \eqref{Sol_A} is said to be \textit{locally upper-H\"{o}lder} at ${a}$ if there exist $\gamma>0,c>0$ and $\varepsilon>0$ such that 
$$\Sol_{\A}(b)\subset \Sol_{\A}(a)+\gamma\|b-a\|^{c}\Ba(0,1)$$ for all $a$ satisfying $\|b- a\|< \varepsilon$, where $\Ba(0,1)$ is the closed unit ball in $\R^n$.

\begin{proposition}%\label{Holder1} 
	If $\A$ is R$_0$ and $\Sol(\A,a)\neq\emptyset$, then the map $\Sol_{\A}$ is  locally upper-H\"{o}lder at $a$.
\end{proposition}

\begin{proof} By assumptions, Corollary \ref{SolA_usc} claims that  $\Sol_{\A}$ is upper semicontinuous at $a$. According to \cite[Theorem 4.1]{LLP2018}, the upper semicontinuity and the local upper-H\"{o}lder stability of $\Sol_{\A}$ at $a$ are equivalent. Hence, the assertion is proved.	\qed
\end{proof}

Let $C$ be a nonempty closed cone. Here $\inte C^+$ stands for the interior of the dual cone $C^+$ of $C$. Note that
$q\in\inte C^+$ if and only if 
$\left\langle v,q \right\rangle >0$ for all $v\in C$ and $v\neq 0$ \cite[Lemma 6.4]{LTY2005}.

\begin{proposition}\label{Holder2} If $\A$ is copositive and $a\in\inte(\Sol(\A,0)^+)$, then
	the map $\Sol_{\A}$ is  locally upper-H\"{o}lder at ${a}$.
\end{proposition}

\begin{proof} Since the upper semicontinuity and the local upper-H\"{o}lder stability of $\Sol_{\A}$ at $a$ are equivalent \cite[Theorem 4.1]{LLP2018}, we need only to prove that  $\Sol_{\A}$ is upper semicontinuous at $a$.
	
Since $\A$ is copositive and $a\in\inte(\Sol(\A,0)^+)$, $\Sol(\A,a)$ is nonempty compact \cite[Corollary 7.3]{Gowda16}. We suppose that $\Sol_{\A}$ is not upper semicontinuous at $a$. There is a nonempty open set $V$ containing $\Sol(\A,a)$ such that, for every integer number $k\geq 1$, there exists $q^k\in \R^n$ satisfying $q^k\to q$ and
	\begin{equation}\label{V_open1}
	x^k\in\Sol(K,P+q^k)\setminus V.
	\end{equation}
	By repeating the argument in the proof of Proposition \ref{usc_1}, we can prove that the sequence $\{x^k\}$ is bounded.	So, without loss of generality we assume that  $x^k\to \bar x$.  It easy to check that  $\bar x\in\Sol(\A,a)$. This leads to $\bar x\in V$. Besides, since $V$ is an open nonempty set, the relation \eqref{V_open1} implies that $\bar x\notin V$. 
	This is a contradiction.  Therefore, $\Sol_{\A}$ is upper semicontinuous at $a$. \qed	
\end{proof}

Theorem 7.5.1 in \cite{CPS1992} shown an interesting result on the stability of the solution map of linear complementarity problems under the copositive condition. Here, we obtain an analogical one for the solution map of tensor complementarity problems. 

\begin{theorem} If $a\in\inte(\Sol(\A,0)^+)$, then there exist constants $\varepsilon>0,\gamma>0$ and $c>0$ such that, if $\B\in\R^{[m,n]}$ and $b\in\R^n$ satisfying
	\begin{equation*}\label{epsilon}
	\max\{\|\B-\A\|,\|b- a\|\}<\varepsilon,
	\end{equation*}
	and $\B$  is copositive, then the following statements hold:
	\begin{description}
	\item[\rm (a)] 	The set $\Sol(\B,a)$ is nonempty and bounded;
	\item[\rm (b)] 	One has
		\begin{equation}\label{ell}\Sol(\B,b)\subset \Sol(\A,a)+\gamma(\|\B-\A\|+\|b-a\|)^{c}\Ba(0,1).	\end{equation}
\end{description}	
\end{theorem}
\begin{proof} $\rm (a)$ We prove that there exists $\varepsilon_1>0$ such that if $\B$  is copositive and $ b\in\R^n$ satisfying \begin{equation}\label{eps_1}
	\max\{\|\B-\A\|,\|b- a\|\}<\varepsilon_1
\end{equation}	 
	then $\Sol(\B,b)$ is nonempty and bounded. Suppose that the assertion is false. There exists a sequence $(\B^k,b^k)$, where $(\B^k,b^k)\to (\A,a)$ and $\B^k$ is copositive,
such that $\Sol(\B^k,b^k)$ is empty or unbounded, for each $k\in\mathbb{N}$. 
From \cite[Corollary 7.3]{Gowda16}, we conclude that $b^k\notin\inte(\Sol(\B^k,0)^+)$. This implies that there exists $x^k\in\Sol(\B^k,0)$ such that $x^k\neq 0$ and $\left\langle x^k,b^k \right\rangle\leq 0$. We can assume that $\|x^k\|^{-1}x^k\to\bar x\in\R^n_+$ with $\|\bar x\|=1$.

Clearly, since $\left\langle x^k,b^k \right\rangle\leq 0$ for each $k\in\mathbb{N}$, one has $\left\langle \bar x,a \right\rangle\leq 0$.
If we prove that $\bar x$ solves $\TCP(\A,0)$, then this contradicts to the assumption that $a\in\inte(\Sol(\A,0)^+)$, and $\rm(a)$ is proved. Thus, we only need to show that $\bar x\in \Sol(\A,0)$. Because $x^k$ belongs to $\Sol(\B^k,0)$, one has
\begin{equation}\label{x_k}
\B^k (x^k)^{m-1}\geq 0, \ \B^k (x^k)^{m} =0.
\end{equation}
By dividing the inequality and the equation in \eqref{x_k} by $\|x^k\|^{m-1}$ and $\|x^k\|^{m}$, respectively, and taking $k\to+\infty$, we obtain 
$$\A \bar x^{m-1}\geq 0, \  \A \bar x^{m}= 0.$$
This implies that $\bar x\in \Sol(\A,0)$.
	
$\rm (b)$ We prove the inclusion \eqref{ell}. According to Proposition \ref{Holder2}, there exist $\gamma_0>0,c>0$ and $\varepsilon$ such that 
	\begin{equation}\label{ell0}
	\Sol(\A,b)\subset \Sol(\A,a)+\gamma_0\|b-a\|^{c}\Ba(0,1)
	\end{equation}
	 for every $b$ satisfying $\|b-a\|< \varepsilon$.
Suppose that $\B$  is copositive and $b\in\R^n$ satisfying \eqref{eps_1}.
For each $z\in \Sol(\B,b)$, by setting 
\begin{equation}\label{q_hat} \hat b=b+\left(\B-\A\right) z^{m-1}, \end{equation}
we have
$$A z^{m-1}+ \hat b=\B z^{m-1}+b\geq 0, \ \;	\left\langle z, \A z^{m-1}+ \hat b\right\rangle =\left\langle z,\B z^{m-1}+b\right\rangle = 0.$$
These show that $z\in\Sol(\A,\hat b)$. Besides, since  $\Sol(\B,b)$ is bounded and nonempty, Remark~\ref{norm_beta} states that there exists $\beta>0$ such that 
	\begin{equation}\label{norm_ineq}
	\|(\B-\A) z^{m-1}\|\leq \beta\|\B-\A\|
	\end{equation}
	for any $z\in\Sol(\B,b)$.  
	From \eqref{eps_1}, \eqref{q_hat}, and  \eqref{norm_ineq}, one has
	$$\begin{array}{ll}
	\|\hat b-a\|   &\leq  \; \|b-a\|+\|(\B-\A)z^{m-1}\| \smallskip \\ 
	&\leq  \; \|b-a\|+\beta \|\B-\A\| \smallskip \\ 
	&\leq \; (1+\beta)\varepsilon_1.
	\end{array}$$
Choosing $\varepsilon_1$ small enough such that $\|\hat b-a\|<\varepsilon$, if necessary.	From \eqref{ell0} and \eqref{norm_ineq}, there exists $x\in\Sol(\A,a)$ such that
	$$\begin{array}{rl}
	\|z-x\| \; & \leq \; \gamma_0\|\hat b-a\|^{c} \smallskip\\
	 &\leq  \; \gamma_0\left(\|b-a\|+\beta \|\B- \A\| \right)^{c} \smallskip    \\ 
	& \leq  \; \gamma \left(\|b-a\|+\|\B-\A\| \right)^{c},
	\end{array} $$
	where $\gamma=\max\left\lbrace \gamma_0^{c},\gamma_0^{c}\beta\right\rbrace $. Then the  inclusion \eqref{ell} is obtained. \qed
\end{proof}

\section{Conclusions} In this paper, we have proved that the set $\Ro$ of all R$_0$--tensors is an open generic semi-algebraic cone . Upper and lower estimates for the dimension of the complement $C(\Ro)$ are shown. Several results on local boundedness, upper semicontinuity and stability of the solution map have been obtained. In our further research, we intend to develop these results for polynomial complementarity problems and semi-algebraic variational inequalities.

\begin{acknowledgements}
The author would like to express his deep gratitude to Professor Nguyen Dong Yen for the enthusiastic encouragement.
\end{acknowledgements}

%=============================================References=================================================%

\end{document}